\documentclass[12pt,letterpaper]{article}

\newcommand{\Keywords}[1]{\par\noindent{\small{\bf Keywords\/}: #1}}
\newcommand{\Class}[1]{\par\noindent{\small{\bf Mathematics Subjects Classification (2020)\/}: #1}}

\usepackage{amsthm, amsmath, amssymb, amsfonts, graphicx, epsfig}
\usepackage{accents}
\usepackage{dsfont}
\usepackage{fancyhdr}
\usepackage{lastpage}
\usepackage[T1]{fontenc}
\usepackage{hyperref}
\usepackage{bbm}
\usepackage{mathrsfs}
\usepackage{ulem}
\usepackage{cancel}

\usepackage{epstopdf}

\usepackage{graphicx}
\usepackage[font=sl,labelfont=bf]{caption}
\usepackage{subcaption}

\usepackage{enumerate}
\usepackage{enumitem}

\usepackage[usenames]{color}

\usepackage{url}
\makeatletter\def\url@leostyle{%
 \@ifundefined{selectfont}{\def\UrlFont{\sf}}{\def\UrlFont{\scriptsize\ttfamily}}} \makeatother

\usepackage[framemethod=default]{mdframed}

\renewcommand\labelenumi{(\arabic*)}
\renewcommand\theenumi\labelenumi

\usepackage[margin=1.01in, letterpaper]{geometry}
\newtheorem{theorem}{Theorem}

\newtheorem{lemma}[theorem]{Lemma}

\theoremstyle{definition}

\newtheorem{example}[theorem]{Example}
\theoremstyle{remark}
\newtheorem{remark}[theorem]{Remark}

\numberwithin{equation}{section}
\numberwithin{theorem}{section}

\definecolor{Red}{rgb}{0.9,0,0.0}
\definecolor{Blue}{rgb}{0,0.0,1.0}

\def\si#1{}
\def\ni#1{}

\def\I{\mathds{1}}
\DeclareMathAlphabet\mathbfcal{OMS}{cmsy}{b}{n}

\def\mypart#1{}

\usepackage{t1enc}	

\DeclareFontFamily{U}{mathx}{\hyphenchar\font45}
\DeclareFontShape{U}{mathx}{m}{n}{
	<5> <6> <7> <8> <9> <10>
	<10.95> <12> <14.4> <17.28> <20.74> <24.88>
	mathx10
}{}
\DeclareSymbolFont{mathx}{U}{mathx}{m}{n}
\DeclareFontSubstitution{U}{mathx}{m}{n}
\DeclareMathSymbol{\bigtimes}{1}{mathx}{"91}
\DeclareMathAccent{\widebar}{0}{mathx}{"73}

\def\cB{\mathcal{B}}

\def\cF{\mathcal{F}}

\def\cX{\mathcal{X}}

\def\bE{\mathbb{E}}

\def\bP{\mathbb{P}}
\def\bQ{\mathbb{Q}}
\def\bR{\mathbb{R}}

\newcommand{\1}{\mathbbm{1}}            
\newcommand{\set}[1]{{\{#1\}}}            

\newcommand{\wt}[1]{\widetilde{#1}}

\def\tilde{\widetilde}

\newtheorem{thm}{Theorem}[section]
\newtheorem{lem}{Lemma}[section]
\newtheorem{pro}{Proposition}[section]
\newtheorem{cor}{Corollary}[section]
\newtheorem{rem}{Remark}[section]
\newtheorem{ex}{Example}[section]
\newtheorem{defi}{Definition}[section]

  \newcommand{\be}{\begin{equation}}
\newcommand{\ee}{\end{equation}}
\newcommand{\bde}{\begin{displaymath}}
\newcommand{\ede}{\end{displaymath}}
\newcommand{\beq}{\begin{eqnarray*}}
\newcommand{\eeq}{\end{eqnarray*}}
\newcommand{\beqa}{\begin{eqnarray}}
\newcommand{\eeqa}{\end{eqnarray}}
\newcommand{\bel }{\left{\begin{array}{ll}}}
\newcommand{\eel}{\cr \end{array} \right.}
\newcommand{\bd}{\begin{defi}}
\newcommand{\ed}{\end{defi}}
\newcommand{\brem }{\begin{rem} \rm }
\newcommand{\erem }{\end{rem}}
\newcommand{\bex}{\begin{ex} \rm }
\newcommand{\eex}{\end{ex}}
\newcommand{\begth}{\begin{thm}}
\newcommand{\eeth}{\end{thm}}
\newcommand{\bl}{\begin{lem}}
\newcommand{\el}{\end{lem}}
\newcommand{\bp}{\begin{pro}}
\newcommand{\ep}{\end{pro}}
\newcommand{\bcor}{\begin{cor}}
\newcommand{\ecor}{\end{cor}}

\usepackage{fancyhdr}
\usepackage{lastpage}
\title{Functional Laws of Large Numbers for Marked Hawkes Processes and Compound Marked Hawkes Processes}

\author{Tomasz R. Bielecki
	\\ Department of Applied Mathematics \\
	Illinois Institute of Technology \\
	Chicago, IL 60616, USA \\ \\
	Jacek Jakubowski \\ University of Warsaw
	\\  Institute     of Mathematics
	\\ Warsaw, Poland \\  \\
	Mariusz Niew\c{e}g\l owski \\
	Faculty of Mathematics and Information Science
	\\ Warsaw University of Technology
	\\ Warsaw, Poland \\ \\
  Anatoliy Swishchuk\\ Department of Mathematics\\ University of Calgary\\ Calgary, Canada}

\makeatletter
\def\and{%
  \end{tabular}%
  \begin{tabular}[t]{c}}%
\def\@fnsymbol#1{\ensuremath{\ifcase#1\or a\or b\or c\or
   d\or e\or f\or g\or h\or i\else\@ctrerr\fi}}
\makeatother

\date{\vskip 30 pt \today \vskip 25 pt}

\begin{document}
\maketitle
\thispagestyle{empty}	
\newpage
\begin{abstract}
We give functional laws of large numbers for a class of marked Hawkes processes and marked compound Hawkes processes with a general mark space. Our results provide some complement to those presented in e.g. \cite{BDHM} and \cite{HORST202194}. As an example we provide an application to analysis of time limit of an insurance ruin process.
\end{abstract}

\vskip 20 pt
\vskip 20 pt
\Class{Primary 60G55, 60G57,  $\,$ Secondary 97M30}



\section{Introduction}

 Functional limit theorems in the theory of stochastic processes are the results that describe the limiting behavior of families of processes as  time  goes to infinity. They are important as they generalize single-point limit theorems to describe the behavior of entire paths of processes. Here, we study a special type of functional limit theorems. Specifically, we study functional laws of large numbers for a family of marked Hawkes processes and some families of marked Hawkes processes.

The object of interest here is  a marked Hawkes process $N=((T_n, X_n))_{n \geq 1}$ with a Borel mark space $(\mathcal{X},\mathbfcal{X} )$.
We associate with the process $N$ an integer-valued random measure on $(\mathbb{R}_+ \times \mathcal{X}, \mathcal{B}(\mathbb{R}_+) \otimes \mathbfcal{X} )$, also denoted by $N$.

The main objective of this paper is to study the limit of $\frac{N((0,vT],A)}{T}$ when $T\rightarrow \infty$ for $v\geq 0$ and $A\in\mathbfcal{X}$. In this regard, the main result of the paper is Theorem \ref{LLN}.

This result relates to Theorem 1 in \cite{BDHM}, which deals with the case of a multivariate Hawkes process without simultaneous excitations. Such a process can be seen as a marked Hawkes process with discrete mark space  $\mathcal{X}$, where marks identify the univariate components of the process. So, our generalization amounts to considering a general Borel mark space. The techniques of our key technical lemmas are in the spirit of the techniques used in \cite{BDHM}. But, because our mark space is not discrete, our proofs are much more intricate.
Theorem \ref{LLN} establishes convergence in the mean square sense, whereas Theorem 1 in \cite{BDHM} provides uniform convergence in the mean square sense. Please see Remark \ref{difference} that sheds light on this difference.

Theorem \ref{LLN} also relates to Theorem 3.6 in \cite{HORST202194}. Our set-up corresponds to the set-up in  \cite{HORST202194} with, using their notation, $N_I$ being a null process and $\mu(\cdot)\equiv m.$ This, in particular, means that  Condition 3.2 in \cite{HORST202194}, assumed in  that paper to prove Theorem 3.6,  is not satisfied in our set-up, and thus  the results of  Theorem 3.6 in \cite{HORST202194} are not applicable here. In particular, it needs to be noted that the techniques used for the proof of Theorem 3.6 in \cite{HORST202194} are different in nature from the techniques the techniques employed in the present paper to prove Theorem \ref{LLN}. To give more insight into the relationship between the present work and \cite{HORST202194} and \cite{BDHM} we note that the results of Theorem 3.6 in \cite{HORST202194} are not applicable in the set-up of \cite{BDHM}. In particular, again, the techniques used for the proof of Theorem 3.6 in \cite{HORST202194} are different in nature from the techniques the techniques employed in  \cite{BDHM} to prove Theorem 1 therein. In this regards we note that in the second paragraph of Section 1.1 in  \cite{HORST202194} the authors comment on equation (1.7) in their paper, which is a consequence of equation (14)  in \cite{BDHM}. The latter equation is the key in \cite{BDHM} to prove  their functional law of large numbers. The authors of \cite{HORST202194} state that there is no obvious way to derive equivalent of equation (14) in \cite{BDHM} to the case of marked Hawkes process when the set of marks is not discrete, and therefore they invent an alternative methodology.  In the present work we actually derive a  counterpart of equation (14) in \cite{BDHM} that is  appropriate in the case of marked Hawkes process with a general, not discrete, mark space. This is done in equation \eqref{eq:ntA} below, and this allows us to generalize to the case of marked Hawkes process with general mark space the proof techniques used in \cite{BDHM}.  Finally, it also needs to be noted that in Section 3.3 of \cite{HORST202194} the authors claim that they can handle the case of constant and  strictly positive  $\mu(\cdot)$ to derive a functional law of large numbers
for processes $N(\cdot,A):=N((0,\cdot],A)$ but no convincing arguments are provided to justify this claim.
In the present work we also derive functional laws of large numbers for two types of compound marked Hawkes processes and provide examples of their applications in insurance. In particular, we show that laws of large numbers for the ruin process given in \cite{StaTor2010}, \cite{Swishchuk2018} and \cite{BJNS2024}, for example, can be derived from our results in a straightforward manner.

The paper is organized as follows: In Section 2 we introduce the main objects of study in this paper. Section 3 develops auxiliary results that are needed for the proof of Theorem 4.1. In Section 4 we state and prove the functional law of large numbers for our marked Hawkes process and for the related compound marked Hawkes processes. Finally, Section 5 provides suggestions for a follow-up research.

\section{Marked Hawkes Process and Compound Marked Hawkes Process}\label{sec:GenGHawkes}

In this section we introduce the objects that we study in this paper.

\subsection{Marked Hawkes Process }

Here we recall the concept of a marked Hawkes process. The version presented below is a special case of the generalized marked Hawkes process introduced in \cite{Bielecki2022a}.

Let $(\Omega,\cF,\bP)$  be  a probability space and $(\mathcal{X},\mathbfcal{X} )$ be a Borel space.
We take $\partial$ to be a point external to $\mathcal{X}$, and we let $\mathcal{X}^\partial := \mathcal{X} \cup \partial$. On $(\Omega,\cF,\bP)$ we consider a marked point process (MPP) $N$ with  mark space $\mathcal{X}$, that is, a sequence of random elements
\begin{equation}\label{eq:Ngen-G}
N=((T_n, X_n))_{n \geq 1},
\end{equation}
where for each $n$:
\begin{enumerate}[label=(\arabic*)]
	\item $T_n$ is a random variable with values in $(0,\infty]$,
	\item $X_n$ is a random variable  with values in $\mathcal{X}^\partial$,
	\item $
	T_n \leq T_{n+1}$, and  if $T_n < + \infty$ then  $
	T_n < T_{n+1}$,
	\item  $ X_n = \partial $ iff $T_n = \infty$.
\end{enumerate}
The explosion time of $N$, say $T_\infty$, is defined as
\[
T_\infty := \lim_{ n \rightarrow \infty } T_n.
\]
Following the typical techniques used in the theory of marked point processes (MPPs) we associate with the process $N$ an integer-valued random measure on $(\mathbb{R}_+ \times \mathcal{X}, \mathcal{B}(\mathbb{R}_+) \otimes \mathbfcal{X} )$, also denoted by $N$ and defined as
\begin{equation}\label{eq:NH-G}
N(dt,dx) := \sum_{n \geq 1 } \delta_{(T_n, X_n)} (dt, dx) \1_\set{T_n < \infty},
\end{equation}
so that
\[
N((0,t],A)=\sum_{n\geq 1} \1_{\{ T_n \leq t\}}\1_{\{X_n \in A \}},
\]
where $A\in\mathbfcal{X}$. We postulate that the corresponding Hawkes kernel is given as
\begin{equation}\label{eq:kappa-G}
\kappa(t,A)=\left (m+\int_{(0, t)\times \mathcal{X} } f(t-s,x)N(ds,dx)\right )\mathbb{Q}(A),
\end{equation}
for $t\geq 0$ and $A\in \mathbfcal{X} $, where $\mathbb{Q}$ is a probability measure on $(\mathcal{X},\mathbfcal{X} )$, $m$ is a non-negative constant, and $f(\cdot,\cdot)$ is a non-negative, bounded measurable function such that
\begin{equation}\label{f}
\int_\cX \| f(\cdot,x)\|^2_{L^1} \bQ (dx)<\infty.
\end{equation}

Proposition 2.7 in \cite{Bielecki2022a} establishes that under this set-up the sequence $X_i,\ i=1,2,\ldots$ is an i.i.d. sequence with the law of $X_1$ equal to $\bQ$. Moreover, the random variables $X_{n+1},\ n=1,2,\ldots$ are independent from  $T_i,\ i=1,2,\ldots,n$ . 

Proposition 2.5 in \cite{Bielecki2022a} gives sufficient conditions on $\kappa$ under which $T_\infty =\infty$ almost surely. In the set-up of the present paper these conditions read that
\begin{itemize}
\item there exist real-valued, continuous and non-negative functions $\beta$ and $\gamma$ such that for $t\geq 0$ we have $f(t-s,x)\leq \beta(t)\gamma(s)$,
\item $\mathbb{E}(\kappa(\cdot,\mathcal{X}))$ is continuous.
\end{itemize}
Both these conditions are satisfied for a large class of kernels $\kappa$, such as the exponential ones for example.

Thus, the random measure $\nu$ defined as
\begin{equation}\label{comp}
\nu(dt,dx)=\kappa(t,dx)dt
\end{equation}
is the compensator of $N$ with respect to the natural filtration of $N$.

We use the notation $N(t,A)$ for $N((0,t],A)$, and $N(t)$ for  $N((0,t],\mathcal{X})$. In particular, $N(0,A)=N(0)=0.$

\subsection{Compound Marked Hawkes Process }
Let $N_t$, $t\geq 0$ be the classical Hawkes process. The classical compound Hawkes process is given as
\[\sum_{n=1}^{N_t} Z_n,\quad t\geq 0,\]
for a sequence $Z_n,\ n=1,2,\ldots$ of real valued random variables (see e.g. \cite{StaTor2010}).

This concept generalizes in a natural way to compound marked Hawkes processes.
We will consider two examples of compound marked Hawkes processes that are motivated by applications in insurance, in particular in cyber-insurance (see e.g. \cite{Zeller2021}) .

\subsubsection{Compound Marked Hawkes Process -- Example 1}
Fix $A\in \mathbfcal{X} $ and let $Z^A_n,\ n=1,2,\ldots$ be a sequence of real valued random variables.    We define process $C^A$ as
\begin{equation}\label{CHA}
  C^A_t=\sum_{n=1}^{N(t,A)}\, Z^A_n =\int_{0}^{t} \zeta^A(s)N(ds,A),\quad t\geq 0,
\end{equation}
where $\zeta^A(s)=Z^A_n$ on $(T_{n-1}, T_n]$.

In the special case with $A=\mathcal{X}$  we simplify the notation to
\begin{equation*}
  C:=C^{\mathcal{X}},\ Z:=Z^{\mathcal{X}},\ \zeta:=\zeta^{\mathcal{X}}.
\end{equation*}
Thus,
\begin{equation}\label{CHX1}
  C_t=\sum_{n=1}^{N(t)}\, Z_n =\int_{0}^{t} \zeta(s)N(ds),\quad t\geq 0.
\end{equation}

\subsubsection{Compound Marked Hawkes Process -- Example 2}
Let $\varphi(x) $ be a real-valued, bounded Borel function. We define process $D^\varphi$ as
\begin{equation}\label{CHA1}
  D^\varphi_t=\int_{(0,t]}\int_{\mathcal{X}} \varphi(x)N(ds,dx)=\sum_{n=1}^{N(t)}\varphi(X_n),\quad t\geq 0.
\end{equation}
This is a special case of \eqref{CHX1} where $Z_n=\varphi(X_n)$.

In the insurance context this would correspond to the amount of claim $\varphi(X_n)$ depending on the type of the claim $X_n$.

\subsection{Thinning representation}
We denote the intensity of $N$ by $\lambda(\cdot)$,
so that
\[\lambda(t)=\kappa(t,\mathcal{X}).\]
\noindent
Let $\mathcal{N}(ds,dx,dz)$ be a time homogeneous Poisson random measure on $(\bR_+ \times \mathcal{X}\times \bR_+, \cB(\bR_+) \otimes  \mathbfcal{X}\otimes \cB(\bR_+) )$, defined on an appropriate extension of our probability space, with intensity $ds\bQ(dx)dz$. Then, we have for a bounded measurable function $h$
\begin{equation*}
\int_{0}^{\infty}\int_{\mathcal{X}}h(s,x)N(ds,dx)=\int_{0}^{\infty}\int_{\mathcal{X}}\int_{0}^{\lambda(s-)}h(s,x)\mathcal{N}(ds,dx,dz)
\end{equation*}
see, e.g., Lemma 1 in \cite{Mas1998}. In particular, for $h(s,x)=I_{\{(0,t]\}}(s)I_{\{A\}}(x)$, $A\in \mathbfcal{X} $,
\begin{equation}\label{thinning-A}
N(t,A)=\int_{\mathcal{X}}I_{\{A\}}(x)N(t,dx)=\int_{0}^{t}\int_{\mathcal{X}}\int_{0}^{\lambda(s-)}I_{\{A\}}(x)\mathcal{N}(ds,dx,dz).
\end{equation}
Consequently,
\begin{equation}\label{thinning-kappa}
\kappa(t,A)=\bQ(A)\left(m+\int_{0}^{t}\int_{\mathcal{X}}\int_{0}^{\lambda(s-)}f(t-s,x)\mathcal{N}(ds,dx,dz)\right ).
\end{equation}
Thus we obtain the representation of kernel \eqref{eq:kappa-G} in terms of  a time homogeneous Poisson random measure $\mathcal{N}$. Hence, in
 particular,
\begin{equation}\label{thinning-lambda}
\lambda(t)=m+\int_{0}^{t}\int_{\mathcal{X}}\int_{0}^{\lambda(s-)}f(t-s,x)\mathcal{N}(ds,dx,dz).
\end{equation}

\section{Auxiliary results}\label{sec:AuxRes}

While proving  auxiliary results we will use some ideas from \cite{HORST202194}.
\\
Towards this end we first define
\begin{equation}\label{F}
F(t)=\int_{\mathcal{X} } f(t,x)\mathbb{Q}(dx),\quad t\geq 0.
\end{equation}
We make the standing assumption that  $F$ is stable:
	\begin{equation}\label{Stab}
	\|F\|_{L^1}<1.
	\end{equation}
 Taking expectation on both sides in \eqref{thinning-lambda}  we get
\begin{align*}
\bE(\lambda(t))&=m+\bE\left (\int_{0}^{t}\int_{\mathcal{X}}\int_{0}^{\lambda(s-)}f(t-s,x)dz\bQ(dx)ds \right )\\
&=m+\int_{0}^{t}\int_{\mathcal{X}} \bE(\lambda(s))f(t-s,x)\bQ(dx)ds \\
&= m+\int_{0}^{t}\bE(\lambda(s))F(t-s)ds.
\end{align*}
So, we obtain the convolution Volterra equation of the second type
 \begin{equation}\label{Elambda}
  \bE(\lambda(t))=m+F\ast\bE(\lambda)(t),\quad t\geq 0,
\end{equation}
where $\ast$ denotes the convolution operator.

Let $r$ be the resolvent of $-F$  defined as the unique solution to the Volterra convolution equation \eqref{F-splot} below (which exists due the assumption \eqref{Stab}; see Theorem 3.1 and Definition 3.2 in Chapter 2 of \cite{Gripenberg1990})
\begin{equation}\label{F-splot}
	r(t)=F\ast r(t)-F(t),\quad t\geq 0.
	\end{equation}
	 Define $R(t)=-r(t),\ t\geq 0.$ Then, the unique solution to \eqref{Elambda} is given as (see Theorem 3.5 in Chapter 2 of \cite{Gripenberg1990} on page 44)
\begin{equation}\label{Elambda-sol}
  \bE(\lambda(t))=m\left (1+\int_{0}^{t}R(t-s)ds\right ),\quad t\geq 0.
\end{equation}

 Now, define
\[F_1(t)=F(t),\quad F_{n+1}(t)=F_n\ast F(t)=\int_{0}^{t}F_n(t-s)F(s)ds,\ n=1,2,3,\ldots .\]
Then
\begin{equation}\label{j-6-1}
	||F_n||_{L^1}=||F||^n_{L^1}.
	\end{equation}
Indeed, for $n=2$ we have
\begin{align*}
||F_{2}||_{L^1}&=\int_{0}^{\infty}dt\int_{0}^{t}F_1(t-s)F(s)ds=\int_{0}^{\infty}dt\int_{0}^{\infty}I_{s\leq t}F(t-s)F(s)ds\\
&=\int_{0}^{\infty}\left (\int_{0}^{\infty}I_{s\leq t}F(t-s)dt\right)F(s)ds=\int_{0}^{\infty}\left (\int_{s}^{\infty}F(t-s)dt\right)F(s)ds\\
&=\int_{0}^{\infty}F(u)du\int_{0}^{\infty}F(s)ds=||F||^2_{L^1}.
\end{align*}
For $n>2$ the proof proceeds accordingly.

Under assumption \eqref{Stab} the series of $F_n,\ n=1,2,3,\dots$ converges in $L^1$ and  we obtain that
\begin{equation}\label{RR}
  R=\sum_{n=1}^{\infty}F_n.
\end{equation}
Since $F_n\geq 0$, then in view of \eqref{j-6-1}  we have  $\|R\|_{L^1}= \sum_{n=1}^{\infty}\|F_n\|_{L^1}=\sum_{n=1}^{\infty}\|F\|^n_{L^1}$. Thus
\begin{equation}\label{R1}\|R\|_{L^1}=\frac{\|F\|_{L^1}}{1-\|F\|_{L^1}}<\infty.
\end{equation}

Note that \eqref{Elambda-sol} and \eqref{R1} imply that
\begin{equation}\label{lambda1}\sup_{t\geq 0} \bE(\lambda(t))<\infty.
\end{equation}

The first lemma will be useful for us, and it generalizes  Lemma 2 in \cite{BDHM} to our set-up.

\begin{lemma}\label{lem:0}
We have
\begin{align}\label{eq:key}
\bE( N(t,A))=\bQ(A) tm +\int_{0}^{t}F(t-s)\bE( N(s,A))ds .
\end{align}
\end{lemma}
\proof
Let
\[\Lambda(r)=\int_{0}^{r}\bE(\lambda(u-))du=\int_{0}^{r}\bE(\lambda(u))du.\]
So, by \eqref{thinning-A} and \eqref{thinning-kappa}, and using Fubini's theorem and definition of $F$ (see \eqref{F}), we conclude that
\begin{align*}
\bE(N(t,A))&=\bE\bigg(\int_{0}^{t}\int_{\mathcal{X}}I_{\{A\}}(x)N(ds,dx)\bigg)=\bE\bigg(\int_{0}^{t}\int_{\mathcal{X}}I_{\{A\}}(x)\kappa(s,dx)ds \bigg) \nonumber \\
&=\bE\bigg(\int_{0}^{t}ds\int_{\mathcal{X}}I_{\{A\}}(x)
\bigg(m+\int_{0}^{s}\int_{\mathcal{X}}\int_{0}^{\lambda(r-)}f(s-r,y)\mathcal{N}(dr,dy,dz)\bigg)\bQ(dx)\bigg)\\
&=\bE\bigg (\int_{0}^{t}ds
\bigg(m+\int_{0}^{s}\int_{\mathcal{X}}\int_{0}^{\lambda(r-)}f(s-r,y)\mathcal{N}(dr,dy,dz)\bigg)\bQ(A)\bigg)\\
&=\bQ(A)mt+\int_{0}^{t}ds
\bigg(\int_{0}^{s}\int_{\mathcal{X}}\bE(\lambda(r))f(s-r,y)dr\bQ(dy)\bigg)\bQ(A)\\
&=\bQ(A)mt+\bQ(A)\int_{0}^{t}ds
\int_{0}^{s}F(s-r)d\Lambda(r).
\end{align*}
Now, again by Fubini's theorem,
\begin{align} \label{j-7-1}
\int_{0}^{t}ds \int_{0}^{s}F(s-r)d\Lambda(r)=\int_{0}^{t}\left(\int_{r}^{t}F(s-r)ds\right) d\Lambda(r)=\int_{0}^{t}\left(\int_{0}^{t-r}F(s)ds\right) d\Lambda(r).
\end{align}
Let
\[\widehat F(t)=\int_{0}^{t}F(s)ds.\]
Then, by integration by parts,
\begin{align}\label{eq:ibp-trick}
0=\Lambda(t)\widehat F(0)-\Lambda(0)\widehat F(t)=\int_{0}^{t}\widehat F(t-r) d\Lambda(r)-\int_{0}^{t}F(t-r) \Lambda(r)dr.
\end{align}
Hence, from \eqref{j-7-1} and the fact that  $\kappa(t,A)=\bQ(A)\lambda(t)$ we obtain
\begin{align*}
&\bQ(A)\int_{0}^{t}ds \int_{0}^{s}F(s-r)d\Lambda(r)
=\int_{0}^{t}F(t-r)\bQ(A)\Lambda(r)dr\\&=\int_{0}^{t}F(t-r)\left(\bQ(A)\int_{0}^{r}\bE(\lambda(u))du\right )dr=\int_{0}^{t}F(t-r)\bE(N(r,A))dr.
\end{align*}
The proof is complete.   \qed

The following two lemmas generalize to our set-up formula (13) in Lemma 4 in \cite{BDHM} and  Lemma 5 for the case  $p=1$ in \cite{BDHM}, respectively.
\begin{lemma}\label{lem:1} Assume \eqref{Stab}. We have
\begin{equation}\label{eq:Lem2-1}
\bE( N(t,A))=\bQ(A)\left (tm+m\int_{0}^{t}sR(t-s)ds\right ).
\end{equation}
\end{lemma}
\proof
Since
\[ \bE( N(t,A))=\bQ(A)\bE\left (\int_0^t \lambda(s)ds\right)\]
it is enough to prove that
\begin{align*}
	\bE\left (\int_0^t \lambda(s)ds\right )=m\left (t+\int_{0}^{t}sR(t-s)ds\right ).
\end{align*}
First, we observe that
\begin{align*}
	\int_{0}^{t} \int_{0}^{s}R(s-r)dr ds =
	\int_{0}^{t} R\ast 1( s)ds = 1\ast R\ast 1(t)=\int_{0}^{t}sR(t-s)ds.
\end{align*}
Next, using \eqref{Elambda-sol} and Fubini's theorem
we have
\begin{align*}\bE\left (\int_0^t \lambda(s)ds\right )&=\int_0^t \bE\left (\lambda(s)\right )ds \\
	&=m\left (t+\int_{0}^{t}\left (\int_{0}^{s}R(s-r)dr\right )ds\right )=m\left (t+\int_{0}^{t}sR(t-s)ds\right ).
\end{align*}
 This finishes the proof. \qed

\begin{lemma}\label{lem:2}
Assume \eqref{Stab} and
\begin{equation}\label{tF}
	\int_{0}^{\infty }t F(t)dt<\infty .
	\end{equation}
We have
\begin{equation}\label{j-9-0}
		T^{-1}\bE( N(Tv,A))-vm\bQ(A)(\|R\|_{L^1}+1)\longrightarrow 0
	\end{equation}
as $T\rightarrow \infty$, uniformly in $v\geq 0$ on any bounded interval.
\end{lemma}
\proof
Using Lemma \ref{lem:1} we have
\begin{align}
&	 vm\bQ(A)(\|R\|_{L^1}+1) -T^{-1}\bE( N(Tv,A)) \notag \\ & =vm\bQ(A)(\|R\|_{L^1}+1)-vm\bQ(A)-m\bQ(A)T^{-1}\int_0^{Tv}tR(Tv-t)dt  \notag \\
	&
=vm\bQ(A)\|R\|_{L^1}-m\bQ(A)T^{-1}\left (Tv\int_0^{Tv}R(Tv-t)dt-\int_0^{Tv}R(Tv-t)(Tv-t)dt\right)  \notag \\
 &=vm\bQ(A)\|R\|_{L^1}-vm\bQ(A)\int_0^{Tv}R(t)dt+m\bQ(A)T^{-1}\int_{0}^{Tv}tR(t)dt  \notag  \\
&=vm\bQ(A)\int_{Tv}^\infty R(t)dt+m\bQ(A)T^{-1}\int_{0}^{Tv}tR(t)dt, \label{j-8-1}
\end{align}
where in the second equation we add and subtract the same factor, and  in the third equation we use the substitution $u=Tv-t$.
Since $R$ is integrable it follows that
\[vm\bQ(A)\int_{Tv}^\infty R(t)dt\longrightarrow 0\]
as $T\rightarrow \infty$, uniformly in $v\geq 0$ on any bounded interval, so the first term in \eqref{j-8-1} goes to zero in the desired manner. Now we will estimate the second term in \eqref{j-8-1}. Straightforward calculations give that
\begin{equation}\label{j-10-1}
		\int_{0}^{\infty} tR(t)dt \leq \frac{\int_{0}^{\infty}tF(t)dt}{(1-\|F\|_{L^1})^2}<\infty.
\end{equation}		
 Finally, since $R\geq 0$, we have for arbitrary $b>0$,\[0\leq \sup_{v\in [0,b]}m\bQ(A)T^{-1}\int_{0}^{Tv}R(t)tdt= m\bQ(A)T^{-1}\int_{0}^{Tb}tR(t)dt\longrightarrow 0\]

\begin{remark}
The result \eqref{j-9-0} can be equivalently stated as
\begin{equation}\label{Lem2-lambda}
\bE \left ( T^{-1}\int_0^{Tv} \lambda(t)dt\right )-vm(\|R\|_{L^1}+1)\longrightarrow 0
\end{equation}
as $T\rightarrow \infty$, uniformly in $v\in [0,b]$ for $b>0$.
Property \eqref{Lem2-lambda} relates to Proposition 3.4 in \cite{HORST202194}. However, the techniques used in \cite{HORST202194} would require that $m=0$ for \eqref{Lem2-lambda} to be satisfied. This would mean that there is no external excitation and therefore diminishing potential applicability of the model. Our techniques employed in this paper allow for $m>0$  and render our model more adequate for applications.

\qed
\end{remark}

Before we proceed recall that
\begin{align}\label{j-9-1}
	\kappa(t,A)=\bQ(A)\lambda(t).
\end{align}
Accordingly, for a fixed $A$, the process $M(t,A),\ t\geq 0,$ given as
\begin{align}\label{M}
	M(t,A)=N(t,A)-\bQ(A)\int_{0}^{t}\lambda(s)ds
\end{align}
is a martingale with jump size equal to 1. Similarly,
	\begin{align}\label{j-9-2}
		\widetilde N(dt,dx):=N(dt,dx)-\bQ(dx)\lambda(t)dt
	\end{align}
is a (local-)martingale measure (cf. e.g. \cite{Kallenberg2017}).

\begin{lemma}\label{lem:predict}
Assume \eqref{Stab}. We have
\begin{align}\label{sup1-predict}
\sup_{ t\in [0,T] }\bE \Big( \Big|	\int_0^t \Big(\int_0^{t-s} F(w) dw
	\Big) d M(s,A) \Big|^2 \Big)
	\leq \bQ(A) \| F \|_{L^1}^2 m(1 + \| R  \|_{L^1}) T
\end{align}
and
\begin{align}\label{sup2-predict}
	&	\sup_{t\in [0,T] } \bE \Big( \Big|\int_0^t \int_\cX \Big(\int_0^{t-u}  f(s,x) ds\Big) \tilde{N}(du,dx) \Big|^2 \Big)
	\leq \int_\cX \| f(\cdot,x)\|^2_{L^1} \bQ (dx) m (1 + \| R  \|_{L^1}) T .
\end{align}
\end{lemma}
\begin{proof}
Fix $L\in [0,T]$ and define process $X_t(L)=\int_0^t \I_{\{s\leq L\}}\Big(\int_0^{L-s} F(w) dw
	\Big) d M(s,A).$  Since $f$ is a bounded function by assumption, then $F$ is bounded as well. Consequently, $\int_0^{L-s} F(w) dw$ is bounded. In view of Proposition I.4.50 c) in \cite{js1987} we see that $X(L)$ is a square integrable martingale. Thus
\begin{align*}
& \bE \Big( \Big|	\int_0^t \I_{\{s\leq L\}}\Big(\int_0^{L-s} F(w) dw
	\Big) d M(s,A) \Big|^2 \Big)
	=
\bE \Big( \int_0^t \I_{\{s\leq L\}}\Big(\int_0^{L-s} F(w) dw
\Big)^2 d N(s,A)  \Big)	
\\
&	=
\bQ(A) \bE \Big( \int_0^t\I_{\{s\leq L\}} \Big(\int_0^{L-s} F(w) dw
\Big)^2  \lambda(s)ds  \Big)	
\leq \bQ(A) \| F \|_{L^1}^2\bE \Big( \int_0^t \lambda(s)ds  \Big)\\
&
\leq \bQ(A) \| F \|_{L^1}^2 m(1 + \| R  \|_{L^1}) t,
\end{align*}	
where in the second equation we use \eqref{j-9-1}, and  in the last inequality we use
\eqref{Elambda-sol}. Hence, taking $L=t$ we obtain
\[\bE \Big( \Big|	\int_0^t \Big(\int_0^{t-s} F(w) dw
	\Big) d M(s,A) \Big|^2 \Big)\leq \bQ(A) \| F \|_{L^1}^2 m(1 + \| R  \|_{L^1}) t
\]
and thus
\begin{align*}
	\sup_{ 0 \leq t
	\leq T }\bE \Big( \Big|	\int_0^t \Big(\int_0^{t-s} F(w) dw
	\Big) d M(s,A) \Big|^2 \Big)
	\leq \bQ(A) \| F \|_{L^1}^2 m(1 + \| R  \|_{L^1}) T,
\end{align*}
which proves \eqref{sup1-predict}. The proof of \eqref{sup2-predict} is done in an analogous way. 	
\end{proof}

Given the above we can prove the following result:

\begin{lemma}\label{lem:L2} Assuming \eqref{Stab} we have
\begin{equation}\label{eq:L2}T^{-1} |N(T,A)-\bE(N(T,A))|\longrightarrow 0
\end{equation}
as $T\rightarrow \infty$, in  $L^2(\bP)$.
\end{lemma}
\begin{proof}
Let $n(t,A):=N(t,A)-\bE(N(t,A))$.
  From \eqref{eq:key} and \eqref{M} we have
\begin{align*}
	n(t&,A)=M(t,A) + \int_0^t \bQ(A) \lambda(s) ds- \bE(N(t,A)) \\
	&=M(t,A) + \int_0^t \bQ(A) \lambda(s) ds-\bQ(A)tm -\int_{0}^{t}F(t-s)\bE( N(s,A))ds
	\\
	&=M(t,A) + \bQ(A) \underbrace{\int_0^t  (\lambda(s) -m) ds}_{=:I_1(t)} -
	\underbrace{\int_{0}^{t}F(t-s) N(s,A)ds}_{=:I_2(t)}
	 +\int_{0}^{t}F(t-s)n(s,A)ds.
\end{align*}
Using integration by parts trick as in \eqref{eq:ibp-trick} and definition of $M$ (see \eqref{M}) we have

\begin{align*}
I_2(t) & =\int_{0}^{t}F(t-s) N(s,A)ds
=\int_0^t \Big(\int_0^{t-s} F(w) dw
\Big) d N(s,A)
\\
&=\int_0^t \Big(\int_0^{t-s} F(w) dw
\Big) d M(s,A) + \bQ(A)\int_0^t \Big(\int_0^{t-s} F(w) dw
\Big) \lambda(s) ds.
\end{align*}
By definition of $\lambda$ and using Fubini's theorem we obtain
\begin{align*}
	I_1(t) & = \int_0^t (\lambda(s) -m) ds
	=
	\int_0^t \Big( \int_0^s \int_\cX f(s-u,x) N(du,dx)\Big)ds \\
	& =
	\int_0^t \int_\cX \Big(\int_u^t  f(s-u,x) ds\Big) N(du,dx)
	.
\end{align*}
Now substituting $I_1$ and $I_2$ into formula for $n(t,A)$, and using \eqref{j-9-2}, yields
\begin{align*}
	n(t,A) &=M(t,A) +\bQ(A) \int_0^t \int_\cX \Big(\int_0^{t-u}  f(s,x) ds\Big) \tilde{N}
	(du,dx)
	 \\
	 & \quad -\int_0^t \Big(\int_0^{t-s} F(w) dw
	 \Big) d M(s,A)
	+\int_{0}^{t}F(t-s)n(s,A)ds .
\end{align*}
So we have
\[
n(t,A ) = J(t,A) + \int_{0}^{t}F(t-s)n(s,A)ds,
\]
where
\begin{align}\label{JTA}
	J(t,A) &:=M(t,A) + \bQ(A) \int_0^t \int_\cX \Big(\int_0^{t-u}  f(s,x) ds\Big) \tilde{N}
(du,dx) \nonumber
\\
& \quad -\int_0^t \Big(\int_0^{t-s} F(w) dw
\Big) d M(s,A).
\end{align}
Thus by the basic renewal theorem (see, e.g., Lemma 3 in \cite{BDHM}) we have
\begin{align}\label{eq:ntA}
n(t,A ) = J(t,A) + \int_{0}^{t}R(t-s)J(s,A)ds.
\end{align}
Therefore,
\begin{align}\label{nice-1}
|n(t,A )| \leq (1 + \| R  \|_{L^1})\sup_{s\in [0,t]}|J(s,A)|.
\end{align}	
Moreover
\begin{align}\label{nice-2}
|J(t,A)|&\leq |M(t,A)| +\bQ(A) \Big|\int_0^t \int_\cX \Big(\int_0^{t-u}  f(s,x) ds\Big) \tilde{N}
(du,dx)\Big| \nonumber
\\
& \quad + \Big|\int_0^t \Big(\int_0^{t-s} F(w) dw
\Big) d M(s,A)\Big|.
\end{align}	
Our goal is to prove
\begin{align}\label{j-13-1}
	\frac{1}{T^2} E|n(T,A )|^2 \rightarrow 0
\end{align}	
when $T\rightarrow \infty$.

Using the Doob's inequality and \eqref{eq:Lem2-1} we get
\begin{align}\label{Doob}
\bE \Big(&\sup_{t \in [0,T]} |M(t,A)|^2\Big)  \leq 4 \bE |M(T,A)|^2 =4 \bE [M(\cdot,A)]_T = 4 \bE N(T,A) \nonumber \\& =4  \bQ(A)(Tm+m\int_{0}^{T}R(T-s)sds)
\leq
4  \bQ(A)m(1+\| R \|_{L^1} ) T,
\end{align}
where the equalities $4 \bE |M(T,A)|^2 =4 \bE [M(\cdot,A)]_T = 4 \bE N(T,A)$ are consequence of  Proposition I.45.c) in \cite{js1987}
since process $N(\cdot,A)$ is of finite variation and $N(0,A)=0$.


Now, using successively \eqref{nice-1}, \eqref{nice-2}, \eqref{Doob}, \eqref{sup1-predict}
and  \eqref{sup2-predict} we conclude that \eqref{j-13-1} is satisfied, which proves the lemma.
\end{proof}

\section{Functional LLN for $N$ and for Compound Marked Hawkes Process}

\subsection{Functional LLN for Process $N$}

In this section we assume that \eqref{Stab} and
\eqref{tF}	hold, and we use the convention that $\frac{0}{0}=0$, so that, for example, the expression $T^{-1} N(Tv,A)$ is defined for $T=0$.

The following theorem corresponds to Theorem 1 in \cite{BDHM}.  The theorem represents a particular instance of functional limit theorems for families of stochastic process. Here, we deal with a family of marked Hawkes processes parameterized by a non-negative real variable $v$. The essence of the theorem is that with the $T^{-1}$ averaging the family converges to a deterministic function of $v$. The parameter $v$ provides rescaling of the physical time, with $0\leq v<1$ modelling stretching the time (time dilation), for example hours to days,   and $v>1$ modeling accelerating the time, for example hours to milliseconds.

\begin{theorem}\label{LLN}
Let us fix $A\in \mathbfcal{X}$. The family of processes $\{(T^{-1} N(Tv,A))_{T\geq 0},\ v\geq 0\}$ (parameterized by  ${v\geq 0}$)
converges to the function $\nu(v)=vm\bQ(A)(\|R\|_{L^1}+1),\ v\geq 0,$ point-wise in $L^2(\bP)$ as  $T\rightarrow \infty$. That is, for each $v\geq 0$,
\[T^{-1} N(Tv,A)-\nu(v)\longrightarrow 0\]
as $T\rightarrow \infty$ in $L^2(\bP)$.
\end{theorem}
\proof  The result is obvious for $v=0$ since $T^{-1}N(0,A)=0=\nu(0)$ for all $T>0$. For $v>0$ it is an immediate consequence of Lemma \ref{lem:2} and Lemma \ref{lem:L2}.
Indeed, we have
\begin{align*}
	&T^{-1} N(Tv,A)-\nu(v) =  \ T^{-1} N(Tv,A)-vm\bQ(A)(\|R\|_{L^1}+1) = \\ &T^{-1}( N(Tv,A) - \bE( N(Tv,A))) +\left (T^{-1}\bE( N(Tv,A))-vm\bQ(A)(\|R\|_{L^1}+1)\right ).
\end{align*}
	So,
\begin{align*}
&\bE\left (T^{-1} N(Tv,A)-vm\bQ(A)(\|R\|_{L^1}+1)\right )^2 \\
& \leq 2\bE\left (T^{-1} N(Tv,A) - \bE( N(Tv,A))\right )^2+2\left (T^{-1} \bE(N(Tv,A))-vm\bQ(A)(\|R\|_{L^1}+1)\right )^2.
\end{align*}
The first term on the right hand side of the above inequality converges to zero as $T\rightarrow \infty $ by Lemma \ref{lem:L2}, and the second term on the right hand side converges to zero as $T\rightarrow \infty $ by Lemma \ref{lem:2}.
\qed

	\begin{remark}\label{difference}
		Here we shed some light on the difference between Theorem \ref{LLN} and Theorem~1 in \cite{BDHM}.
		The marked Hawkes process with a finite set of marks considered in \cite{BDHM} is cast there in the format of an unmarked multivariate Hawkes process. An unmarked Hawkes process can be seen as a marked Hawkes process with $\mathcal{X}$ being a singleton, say $\mathcal{X}=\set{0}$. If in our set-up we take $\mathcal{X}=\set{0}$, then
		\[M(\cdot,\set{0})=\wt N(\cdot,\set{0}).\]
		Consequently, using \eqref{JTA} we obtain
		\[J(\cdot,\set{0})=M(\cdot,\set{0}),\]
		and in this case, invoking \eqref{nice-1}, we obtain
		\begin{align}\label{nice-1-sing}
			\sup_{t\in [0,T]}|n(t,\set{0} )| \leq (1 + \| R  \|_{L^1})\sup_{t\in [0,T]}|M(t,\set{0})|.
		\end{align}
		Hence and from \eqref{Doob}	 we have
			\begin{align}
			\bE \Big(&\sup_{t \in [0,T]} |n(t,\set{0})|^2\Big)  \leq CT
		\end{align}
		for a finite constant $C>0$.  This implies that
		\begin{equation}\label{eq:L2-V1}
			T^{-1} |N(Tv,\set{0})-\bE(N(Tv,\set{0}))|\longrightarrow 0
		\end{equation}
		as $T\rightarrow \infty$, in  $L^2(\bP)$ uniformly in $v$ on any bounded interval, which in turn, jointly   with Lemma \ref{lem:2}, implies that
		\[T^{-1} N(Tv,\set{0})-vm(\|R\|_{L^1}+1)\longrightarrow 0\]
		as $T\rightarrow \infty$ in $L^2(\bP)$ uniformly in $v$ on any bounded interval, which agrees with what Theorem~1 in \cite{BDHM} states. However, since the mark space $(\mathcal{X},\mathbfcal{X} )$ considered in this paper is a general one, the results of Theorem 1 in \cite{BDHM}  do not apply here, and thus Theorem 4.1 can not be cast into the framework considered in \cite{BDHM}.	\end{remark}

\subsection{Functional LLN for Process $C^A$}
Recall that for $A\in \mathbfcal{X}$ and a sequence $Z^A_n,\ n=1,2,\ldots$,  of real valued random variables the process $C^A$ is defined as
\begin{equation*}
  C^A_t=\sum_{n=1}^{N(t,A)}\, Z^A_n =\int_{0}^{t} \zeta^A(s)N(ds,A),\quad t\geq 0,
\end{equation*}
where  $\zeta^A(s)=Z^A_n$ on $(T_{n-1}, T_n]$.
Assume that $Z^A_n,\ n=1,2,\ldots$ is an i.i.d. sequence of random variables with finite first moment and independent from $N$.
We then have the following
	\begin{theorem}\label{LLN-CA}
The family of processes $\{(T^{-1}C^A_{Tv})_{T\geq 0},\ v\geq 0\}$
(parameterized by  ${v\geq 0}$) converges to the function $\upsilon(v)=vm\bQ(A)(\|R\|_{L^1}+1)\bE(Z^A_1),\ v\geq 0,$ point-wise in probability as  $T\rightarrow \infty$. That is, for each $v\geq 0$, we have
\[\frac{C^A_{Tv}}{T}\longrightarrow \upsilon(v)\]
as $T\rightarrow \infty$, in probability.
\end{theorem}

\proof
The result is obvious for $v=0$ (c.f. the proof of Theorem 4.1). Thus, we take $v>0$ in the rest of the proof.
We have
\[\frac{C^A_{Tv}}{T}=\frac{\sum_{n=1}^{N(Tv,A)}\, Z^A_n}{N(Tv,A)}\frac{N(Tv,A)}{T}.\]
It follows from Theorem \ref{LLN} that
\[N(Tv,A)\longrightarrow \infty\]
in probability, when $T\rightarrow \infty$.
Consequently, combining this with the strong law of large numbers for the sequence $Z^A_n,\ n=1,2,\ldots$, and using  Theorem 2.2 in \cite{Gut2009} we conclude that
\[\frac{\sum_{n=1}^{N(Tv,A)}\, Z^A_n}{N(Tv,A)}\longrightarrow \bE(Z^A_1)\]
as $T\rightarrow \infty$, in probability.
 It also follows from Theorem \ref{LLN} that
\[\frac{N(Tv,A)}{T}\longrightarrow vm\bQ(A)(\|R\|_{L^1}+1)\]
in probability as $T\rightarrow \infty$.
 Thus, using the above and applying the continuous mapping theorem we conclude that
\[\frac{C^A_{Tv}}{T}\longrightarrow vm\bQ(A)(\|R\|_{L^1}+1)\bE(Z^A_1)\]
in probability, as $T\rightarrow \infty$.\qed 

\begin{example}[Application in insurance]
  As it has been observed, clustering and self-exciting arrivals of claims is an inherent feature of claims arrival processes that is faced by insurance companies. See e.g. \cite{DasZha2012} or \cite{Magnusson2015ARS}.  Thus, modeling claims arrivals in terms of a marked Hawkes process is fully justified and supported by empirical evidence.

Consider an insurance company. Let $N$ be a marked Hawkes process that models the claims arrivals.  We take discrete mark space, say $\mathcal{X}=\set{x_k,\, k=1,2,\ldots,K}$, where $x_k$s represent type/severity of arriving claims.

Now, let $Z^{(k)}_n,\ n=1,2,\ldots$, be i.i.d. random variables, independent of $N$, representing the claim sizes triggered by the claim of type/severity $x_k$.  A classical problem in insurance is the problem of long-time behavior of the surplus process.

The surplus process, say $\mathcal{R}$, is  given here as the following extension of the classical Cramer-Lundberg model  and termed a compound marked Hawkes process:
\begin{equation}
\mathcal{R}_t=r+c t -\sum_{k=1}^{K}\sum_{n=1}^{N((0,t],\{x_k\})}\, Z^{(k)}_n,
\end{equation}
where $r$ is the initial capital and $c$ is the insurance premium rate. This model generalizes the compound Hawkes surplus model originally introduced in \cite{StaTor2010}.
Now, let us set $\mu^k=\bE(Z^{(k)}_1)$ for $k=1,2,\ldots,K.$ Then, Theorem \ref{LLN-CA} gives us that  for $v\geq 0$
\begin{equation}\label{ruin}
\frac{\mathcal{R}_{Tv}}{T}\longrightarrow v\left (c- m(\|R\|_{L^1}+1)\sum_{k=1}^{K}\mu^k\bQ(\{x_k\})\right )
\end{equation}
as $T\rightarrow \infty$, in probability.
\end{example}
\begin{remark}
Consider a special case of the above example where the claim sizes do not depend on the mark. Thus, $Z^{(k)}_n=Z_n=Z^{\mathcal{X}}_n$ for $k=1,\ldots,K$, $n=1,2,\ldots,$ where $Z_n,\ n=1,2,\ldots$ is a  sequence of  i.i.d. random variables. So, here we have
\begin{equation}
\mathcal{R}_t=r+c t -\sum_{k=1}^{K}\sum_{n=1}^{N((0,t],\{x_k\})}\, Z_n=r+c t -\sum_{n=1}^{N(t)}\, Z_n=r+ct -C_t.
\end{equation}
Let $\mu:=\bE(Z_1)=\bE(Z^{\mathcal{X}})$. Accordingly, using Proposition \ref{LLN-CA} and taking $v=1$ we obtain the following special case of the LLN for the ruin process
\begin{equation}\label{ruin-special}
\frac{\mathcal{R}_{T}}{T}\longrightarrow c- \mu m(\|R\|_{L^1}+1)
\end{equation}
as $T\rightarrow \infty$, in probability.

We will now show that this agrees with the classical LLN for a ruin process driven by a compound marked Hawkes process, as presented in \cite{StaTor2010}, \cite{Swishchuk2018} or \cite{BJNS2024}, for example.
Towards this end we set (assuming integrability)
\[H(x)=\int_{0}^{\infty}f(t,x)dt,\quad  x\in \mathcal{X}.\]
Thus, using Fubini's theorem,
\[{\mathbb{E}H(X_1)=\int_{\mathcal{X}}H(x)\mathbb{Q}(dx)=\int_{\mathcal{X}}\left [\int_{0}^{\infty}f(t,x)dt\right ]\mathbb{Q}(dx)}\]
\[{=\int_{0}^{\infty}\left [\int_{\mathcal{X}}f(t,x)\mathbb{Q}(dx)\right ]dt=\int_{0}^{\infty}F(t)dt=||F||_{L^1}.}\]
The classical stability condition for marked Hawkes processes is that $\mathbb{E}H(X_1)<1$ which, in view of the above, is the same as \eqref{Stab}. Finally, also using \eqref{R1} and \eqref{ruin-special}, we get that
\begin{align}\label{ruin-special-1}
\frac{\mathcal{R}_{T}}{T}\longrightarrow \ & c- \mu m(\|R\|_{L^1}+1)=c- \frac{\mu m}{1-||F||_{L^1}}= c- \frac{\bE(Z_1) m}{1-\mathbb{E}H(X_1)}
\end{align}
as $T\rightarrow \infty$, which is a version of the classical LLN for a ruin process driven by a compound marked  Hawkes process.

An important condition for insurance risk management is the so called net-profit condition (c.f. \cite{Rolski1999}). In our set-up the net-profit condition takes the form
\[
	 c> \frac{\bE(Z_1) m}{1-\mathbb{E}H(X_1)}.
\]
Recall that if the net profit condition is violated then ruin is inevitable.  Taking $f(\cdot,\cdot) \equiv 0$ and $m>0$ yields the classical Crammer-Lundberg model.  In this case  $\mathbb{E}H(X_1)=0$ and we recover
the classical net-profit condition
\[
c> \bE(Z_1) m.
\]
Since $\frac{\bE(Z_1) m}{1-\mathbb{E}H(X_1)}> \bE(Z_1) m$ we see that the net-profit condition is more restrictive in  the presence of the self-excitation effect in arrival of claims. This means that if the insurer recognizes  the change of the structure of claims arrivals (from Poisson model to Hawkes model) then in order to prevent inevitability of ruin they should increase the premium rate $c$ above $\frac{\bE(Z_1) m}{1-\mathbb{E}H(X_1)}$ if needed.   \qed
\end{remark}

\subsection{Functional LLN for Process $D^\varphi$}
Recall that for  a real-valued, bounded Borel function  $\varphi(x) $ the process $D^\varphi$ is defined as
\begin{equation*}
  D^\varphi_t=\int_{(0,t]}\int_{\mathcal{X}} \varphi(x)N(ds,dx)=\sum_{n=1}^{N(t)}\varphi(X_n),\quad t\geq 0.
\end{equation*}

We begin by observing  that from \eqref{j-9-1} and Lemma \ref{lem:1} it follows that
\[\bE (D^\varphi_t) = \bE  \left(\sum_{n=1}^{N(t)}\varphi(X_n)\right ) = \varphi_\bQ \bE \left (\int_{(0,t)}\lambda (s)ds\right )=
		\varphi_\bQ m  \left(t+\int_{0}^{t}R(t-s)sds \right ),\]
where $\varphi_\bQ:=\int_{\mathcal{X}}\varphi(x)\bQ(dx).$

 \begin{theorem}\label{Dvarphi} Assume that the following integrability condition is satisfied
\begin{equation}\label{assume}
\int_{\mathcal{X}}\left [\int_{0}^{\infty}\left(f(u,x)+\int_{0}^{u} R(u-r)f(r,x)dr\right )du\right ]^2\bQ(dx)<\infty.
\end{equation}
 Then, the family of processes $\{(T^{-1}D^\varphi_{vT})_{T\geq 0},\ v\geq 0\}$ (parameterized by $v\geq 0$)
 converges to the function $\psi(v)=\varphi_\bQ vm (\|R\|_{L^1}+1),\ v\geq 0,$ point-wise in probability as  $T\rightarrow \infty$. That is,
\begin{equation}\label{j-17-1}
	\frac{D^\varphi_{vT}}{T}\longrightarrow \psi(v)
\end{equation}
as $T\rightarrow \infty$ in probability, for any $v\geq 0.$
\end{theorem}

\proof
 To begin we write $D^\varphi$ using $\widetilde{N}$ (see \eqref{j-9-2})
\[D^\varphi_t=\int_{0}^{t}\int_{\mathcal{X}}\varphi(x) \widetilde{N}(ds,dx)+\varphi_\bQ\int_{0}^{t}\lambda(s)ds.\]
Using the Doob's maximal inequality and \eqref{lambda1} we get for any $V\in (0,\infty)$
\[\bE\left (\sup_{v\in [0,V]}\left (T^{-1}\int_{0}^{vT}\int_{\mathcal{X}}\varphi(x) \widetilde{N}(ds,dx)\right )^2\right )\leq \frac{C}{T} \int_{0}^{V}\bE(\lambda (Ts))ds\int_{\mathcal{X}}\varphi^2 (x)\bQ(dx)\leq \frac{C'}{T},\]
where $C$ and $C'$ are constants not depending on $T$. Thus, $\sup_{v\in [0,V]}T^{-1}\int_{0}^{vT}\int_{\mathcal{X}}\varphi(x) \widetilde{N}(ds,dx)$ converges to zero in probability as $T \rightarrow \infty.$
Thus, to prove \eqref{j-17-1} it is enough to prove
\begin{equation}\label{j-17-2}
\frac{\varphi_\bQ \int_{0}^{Tv}\lambda(s)ds}{T}\longrightarrow \psi(v)	
\end{equation}
in probability as $T\rightarrow \infty$ (for any $v\geq 0$).
We start with showing that
\begin{equation}\label{j-17-3}
	T^{-1}\left |\int_{0}^{Tv}\lambda(s)ds-\int_{0}^{Tv}\bE(\lambda(s))ds\right | \longrightarrow 0
	\end{equation}
 in probability for any $v\geq 0.$  Towards this end we proceed as follows. Denote $\wt {\mathcal{N}}(ds,dx,dz)=\mathcal{N}(ds,dx,dz)-ds\bQ(dx)dz$, which is a (local) martingale measure.
Using the representation  \eqref{thinning-lambda} of $\lambda$ we obtain
\begin{equation}\label{thinning-lambda-tilde}
\lambda(t)=F\ast \lambda(t)+ m+\int_{0}^{t}\int_{\mathcal{X}}\int_{0}^{\lambda(s-)}f(t-s,x)\wt {\mathcal{N}}(ds,dx,dz) .
\end{equation}
and thus
\begin{equation*}\label{thinning-lambda-tilde-ast}
R\ast\lambda(t)=R\ast (F\ast \lambda)(t)+ R\ast m(t)+\int_{0}^{t}R(t-s)\left (\int_{0}^{s}\int_{\mathcal{X}}\int_{0}^{\lambda(r-)}f(s-r,x)\wt {\mathcal{N}}(dr,dx,dz)\right )ds.
\end{equation*}
From \eqref{F-splot}  we have that $R$ satisfies $R = F + R \ast F $ and thus  $(F+R\ast F)\ast \lambda(t)=R\ast \lambda(t)$, so the above gives
\begin{equation*}\label{thinning-lambda-tilde-ast-1}
F\ast\lambda(t)= R\ast m(t)+\int_{0}^{t}R(t-s)\left (\int_{0}^{s}\int_{\mathcal{X}}\int_{0}^{\lambda(r-)}f(s-r,x)\wt {\mathcal{N}}(dr,dx,dz)\right )ds.
\end{equation*}
Thus, from this and from \eqref{thinning-lambda-tilde} we obtain
\begin{align}\label{thinning-lambda-tilde-ast}
\lambda(t)&=m+ R\ast m(t)+\int_{0}^{t}\int_{\mathcal{X}}\int_{0}^{\lambda(s-)}f(t-s,x)\wt {\mathcal{N}}(ds,dx,dz) \nonumber \\ &+\int_{0}^{t}R(t-s)\left (\int_{0}^{s}\int_{\mathcal{X}}\int_{0}^{\lambda(r-)}f(s-r,x)\wt {\mathcal{N}}(dr,dx,dz)\right )ds \nonumber \\
&=m+ R\ast m(t)+\int_{0}^{t}\int_{\mathcal{X}}\int_{0}^{\lambda(s-)}f(t-s,x)\wt {\mathcal{N}}(ds,dx,dz) \nonumber \\
&+\int_{0}^{t}\int_{\mathcal{X}}\int_{0}^{\lambda(s-)}\int_{0}^{t-s}R(t-s-r)f(r,x)dr\wt {\mathcal{N}}(ds,dx,dz)\nonumber \\
&=m+ R\ast m(t)\nonumber \\
&+\int_{0}^{t}\int_{\mathcal{X}}\int_{0}^{\lambda(s-)}\left(f(t-s,x)+\int_{0}^{t-s} R(t-s-r)f(r,x)dr\right )\wt {\mathcal{N}}(ds,dx,dz),
\end{align}
where the second equality follows by applying Fubini's theorem to \[\int_{0}^{t}R(t-s)\left (\int_{0}^{s}\int_{\mathcal{X}}\int_{0}^{\lambda(r-)}f(s-r,x)\wt {\mathcal{N}}(dr,dx,dz)\right )ds.\]
Now, from \eqref{thinning-lambda-tilde-ast} and \eqref{Elambda-sol} it follows that
\begin{align}\label{lambda-Elambda}
&\int_{0}^{Tv}\lambda(t)dt-\int_{0}^{Tv}\bE(\lambda(t))dt \nonumber \\ &=\int_{0}^{Tv}\int_{0}^{t}\int_{\mathcal{X}}\int_{0}^{\lambda(s-)}\left(f(t-s,x)+\int_{0}^{t-s} R(t-s-r)f(r,x)dr\right )\wt {\mathcal{N}}(ds,dx,dz)dt \nonumber \\
&=\int_{0}^{v}\int_{\mathcal{X}}\int_{0}^{\lambda(Ts-)}\left [\int_{0}^{T(v-s)}\left(f(u,x)+\int_{0}^{u} R(u-r)f(r,x)dr\right )du\right ]\wt {\mathcal{N}}(Tds,dx,dz).
\end{align}
Given the above we have
\begin{align}\label{E-lambda-Elambda}
&\bE \left (\int_{0}^{Tv}\lambda(t)dt-\int_{0}^{Tv}\bE(\lambda(t))dt \right )^2 \nonumber
\\
&
\leq \int_{0}^{v}\bE(\lambda(Ts))ds\int_{\mathcal{X}}\left [\int_{0}^{\infty}\left(f(u,x)+\int_{0}^{u} R(u-r)f(r,x)dr\right )du\right ]^2\bQ(dx).
\end{align}
Thus, using \eqref{lambda1}  and  \eqref{assume} we conclude that
\[\lim_{T\rightarrow \infty }\frac{1}{T}\bE\left (\int_{0}^{Tv}\lambda(t)dt-\int_{0}^{Tv}\bE(\lambda(t))dt\right )^2 =0,\]
which gives \eqref{j-17-3}. Since
\begin{align} \label{j-18-2}
	\bE\left (\int_0^t \lambda(s)ds\right )=m\left (t+\int_{0}^{t}sR(t-s)ds\right )
\end{align}
(see Lemma \ref{lem:1}) and
\begin{align*}
	\int_{0}^{Tv}sR(Tv-s)ds = - \int_{0}^{Tv}(Tv-s)R(Tv-s)ds + \int_{0}^{Tv}TvR(Tv-s)ds,
\end{align*}
using \eqref{j-10-1} (which is a consequence of assumption \eqref{tF})
we obtain that
\begin{align} \label{j-18-1}
	\int_{0}^{Tv}sR(Tv-s)ds \longrightarrow \psi(v)
\end{align}
as $T\to \infty$.
Finally, from \eqref{j-17-3}, \eqref{j-18-2}, \eqref{j-18-1} we obtain \eqref{j-17-2} and we conclude the proof.

\begin{remark}
The distinctive difference between Theorem \ref{LLN-CA} and Theorem \ref{Dvarphi} is that in the former one it was assumed that the sequence $Z^A_n$, $n=1,2,...$ is independent of the process $N$, whereas in case of the latter one the sequence $\varphi(X_n)$, $n=1,2,...$ and the process $N$ are not independent.
\end{remark}

\section{Future research}

The functional laws of large numbers presented here do not give uniform convergence. Please see Remark \ref{R1} in this regard. We plan to study the uniform convergence in a follow-up work.

As stated above, our Theorem \ref{LLN} relates to Theorem 1 in \cite{BDHM}. In the latter theorem the authors also consider the almost sure convergence, the case that is not studied in the present paper. We plan to study this mode of convergence for our marked processes in a follow-up work as well.

Finally, we plan to complement the present results with functional central limit theorems for marked Hawkes processes and marked compound Hawkes processes with a general mark space.

\section{Acknowledgement}   The authors thank to BIRS, Banff, AB, Canada, for providing very friendly and excellent environment for our fruitful and productive work  during the workshop "Applications of Multivariate Hawkes Processes in Finance, Insurance and Epidemiology" which have resulted in several papers, including this one. The authors also thank the University of Warsaw which supported this research via grant IDUB-POB3-D110-003/2022. A. Swishchuk thanks to NSERC for continuing support. \\ The authors would also like to thank the two anonymous referees for very careful reading of the manuscript and their comments and remarks that allowed us to improve the presentation of the results.

				\bibliographystyle{alpha}
				\bibliography{Math_Fin,Math_Fin_1,Hawkes}

\end{document}